\documentclass[a4paper,10pt]{article}

\usepackage[utf8]{inputenc}
\usepackage[english]{babel}
\usepackage{amsmath,amssymb,amsthm} % amsfonts
\usepackage{hyperref,xcolor,fullpage}
\usepackage{titlesec}

\makeatletter \let\@fnsymbol\@arabic \makeatother % don't like that cross next to names
\titleformat{\section}[block]
	{\normalfont\large\bfseries\centering}
	{\thesection}{.5em}{}
\titleformat{\subsection}[runin]
	{\normalfont\bfseries}
	{\thesubsection}{.5em}{}
\titleformat{\subsubsection}[runin]
	{\normalfont\bfseries}
	{\thesubsubsection}{.5em}{}
 
\allowdisplaybreaks % allows break in multiline formulas
\usepackage[numbers]{natbib} 

\newtheorem{theorem}{Theorem}[section]
\newtheorem{lemma}[theorem]{Lemma}

\newtheorem{corollary}[theorem]{Corollary}

\newtheorem{proposition}[theorem]{Proposition}
\newtheorem{remark}[theorem]{Remark}

\makeatletter
\let\save@mathaccent\mathaccent
\newcommand*\if@single[3]{
	\setbox0\hbox{${\mathaccent"0362{#1}}^H$}%
	\setbox2\hbox{${\mathaccent"0362{\kern0pt#1}}^H$}%
	\ifdim\ht0=\ht2 #3\else #2\fi }
\newcommand*\rel@kern[1]{\kern#1\dimexpr\macc@kerna}
\newcommand*\widebar[1]{\@ifnextchar^{{\wide@bar{#1}{0}}}{\wide@bar{#1}{1}}}
\newcommand*\wide@bar[2]{\if@single{#1}{\wide@bar@{#1}{#2}{1}}{\wide@bar@{#1}{#2}{2}}}
\newcommand*\wide@bar@[3]{
	\begingroup
	\def\mathaccent##1##2{
		\let\mathaccent\save@mathaccent
		\if#32 \let\macc@nucleus\first@char \fi
		\setbox\z@\hbox{$\macc@style{\macc@nucleus}_{}$}
		\setbox\tw@\hbox{$\macc@style{\macc@nucleus}{}_{}$}
		\dimen@\wd\tw@ \advance\dimen@-\wd\z@ \divide\dimen@ 3 \@tempdima\wd\tw@ \advance\@tempdima-\scriptspace \divide\@tempdima 10 \advance\dimen@-\@tempdima \ifdim\dimen@>\z@ \dimen@0pt \fi \rel@kern{0.6}\kern-\dimen@
		\if#31 \overline{\rel@kern{-0.6}\kern\dimen@\macc@nucleus\rel@kern{0.4}\kern\dimen@} \advance\dimen@0.4\dimexpr\macc@kerna \let\final@kern#2 \ifdim\dimen@<\z@ \let\final@kern1 \fi
		\if \final@kern1 \kern-\dimen@ \fi
		\else \overline{\rel@kern{-0.6}\kern\dimen@#1} \fi }
	\macc@depth\@ne	\let\math@bgroup\@empty \let\math@egroup\macc@set@skewchar 	\mathsurround\z@ \frozen@everymath{\mathgroup\macc@group\relax} 	\macc@set@skewchar\relax \let\mathaccentV\macc@nested@a	\if#31 \macc@nested@a\relax111{#1} \else \def\gobble@till@marker##1\endmarker{} \futurelet\first@char\gobble@till@marker#1\endmarker \ifcat\noexpand\first@char A\else \def\first@char{} \fi \macc@nested@a\relax111{\first@char} \fi
	\endgroup }
\makeatother

\newcommand{\PP}[1]{\mathbb{P}\left(#1\right)}

\newcommand{\EE}[1]{\mathbb{E}\!\left(#1\right)}

\newcommand{\Var}[1]{\text{Var}\!\left(#1\right)}

\newcommand{\ba}{\textbf{a}}
\newcommand{\bb}{\textbf{b}}
\newcommand{\bc}{\textbf{c}}
\newcommand{\bx}{\textbf{x}}
\newcommand{\by}{\textbf{y}}

\usepackage{mathrsfs}
\DeclareMathAlphabet{\mathpzc}{OT1}{pzc}{m}{it}
\usepackage{xifthen}
\newcommand{\fp}[1][]{ \ifthenelse{\isempty{#1}}{\mathpzc{p}}{\mathpzc{p}(#1)} }
\newcommand{\fP}{\mathpzc{P}}

\newcommand{\mM}{\mathcal{M}}
\newcommand{\mS}{\mathcal{S}}
\newcommand{\mH}{\mathcal{H}}

\newcommand{\mF}{\mathcal{F}}

\newcommand{\mT}{\mathcal{T}}

\newcommand{\mI}{\mathcal{I}}

\newcommand{\NN}{\mathbb{N}}
\newcommand{\ZZ}{\mathbb{Z}}

\newcommand{\FF}{\mathbb{F}}

\newcommand{\rank}[1]{\text{rk} ( {#1} )}
\newcommand{\Sub}[2]{{#1} [{#2}]}

\newcommand{\floor}[1]{\lfloor {#1} \rfloor}

\definecolor{darkred}{cmyk}{.3,.9,.80,.2}
\title{A Note on Sparse Supersaturation and Extremal Results\\ for Linear Homogeneous Systems}
\date{}
\author{	Christoph Spiegel \thanks{Universitat Polit\`ecnica de Catalunya and Barcelona Graduate School of Mathematics, Department of Mathematics, Edificio Omega, 08034 Barcelona, Spain. E-mail: {\tt christoph.spiegel@upc.edu}. Supported by the Spanish Ministerio de Econom\'{i}a y Competitividad FPI grant under the project MTM2014-54745-P.}}

\begin{document}
\maketitle

\begin{abstract}
	We study the thresholds for the property of containing a solution to a linear homogeneous system in random sets. We expand a previous sparse Szémeredi-type result of Schacht to the broadest class of matrices possible. We also provide a shorter proof of a sparse Rado result of Friedgut, Rödl, Ruciński and Schacht based on a hypergraph container approach due to Nenadov and Steger. Lastly we further extend these results to include some solutions with repeated entries using a notion of non-trivial solutions due to Rúzsa as well as Rué et al.
\end{abstract}

\section{Introduction}

A \emph{$k$-term arithmetic progression} is a set of integers that can be written as $\{ a, a+d, a+(k-1)d \}$ for some $a, d, k \in \ZZ$, $k \geq 3$ and $d \neq 0$.  The Theorem of van der Waerden~\cite{vdW27} states that every finite colouring of $[n] = \{1, \dots , n \}$ contains a monochromatic $k$-term arithmetic progression for $n$ large enough. Szemerédi's Theorem~\cite{Sz75} strengthened this result by stating that every set of integers with positive natural density contains a $k$-term arithmetic progression. Rado~\cite{Ra33} generalized van der Waerden's result to certain systems of linear equations and Frankl, Graham and Rödl~\cite{FGR88} did the same for Szemerédi's extremal result.

A common area of interest is to study 'sparse' or 'random' versions of such results. Consider the \emph{binomial random set} $[n]_p$ where each element in $[n]$ is chosen independently with probability $p = p(n)$. That is $[n]_p$ is a random variable sampling from the finite probability space on all subsets of $[n]$ that assigns each $T \subseteq [n]$ the probability $\PP{[n]_ p = T} = p^{|T|} (1-p)^{n-|T|}$. Given an integer valued matrix $A \in \mM_{r \times m}(\ZZ)$ with $r$ rows and $m$ columns, we largely follow the notation of Schacht~\cite{Sch12} and let $\mS(A) = \{ \bx \in \ZZ^m : A \cdot \bx^T = \mathbf{0}^T \}$ be the set of all solutions and $\mS_0(A) = \{ \bx = (x_1, \dots , x_m) \in \mS(A) : x_i \neq x_j \text{ for } i \neq j \}$ the set of all \emph{proper} solutions. We call $A$ \emph{irredundant} if $\mS_0(A) \neq \emptyset$. Given a set of integers $T$ and $s \in \NN$ we write
\begin{equation}
	T \to_s A
\end{equation}
if for every finite partition $T_1 \dot{\cup} \dots \dot{\cup} \, T_s = T$ there exists $1 \leq i \leq s$ such that $T_i \cap \mS_0(A) \neq \emptyset$. An irredundant matrix $A \in \mM_{r \times m}(\ZZ)$ is \emph{partition regular} if for every $s \in \NN$ we have $[n] \to_s A$ for $n$ large enough. Rado~\cite{Ra33} gave the \emph{column condition} as a characterization of partition regular matrices.

Rödl and Ruci\'nski~\cite{RR97} formulated a sparse version of Rado's Theorem that was later completed by Friedgut, Rödl and Schacht~\cite{FRS10}. To state it let $\emptyset \neq Q \subseteq [m]$ be any set of column indices and define $r_Q = \rank{A} - \rank{A^{\widebar{Q}}}$ where $A^{\widebar{Q}}$ is the matrix obtained by keeping only the columns which are indexed by $\widebar{Q}$ and the rank of a matrix $A$ is denoted as $\rank{A}$. Here $A^{\emptyset}$ is the empty matrix with $ \rank{A^{\emptyset}} = 0$. The \emph{maximum $1$-density} of a given matrix $A \in \mM_{r \times m}(\ZZ)$ is defined as
	\begin{equation} \label{eq:max1density}
		m_1(A) = \max_{\substack{Q \subseteq [m] \\ 2 \leq |Q|}} \frac{|Q|-1}{|Q|-r_Q-1}.
	\end{equation}
	We will later see that for the specific kinds of matrices under consideration, this is indeed well-defined, that is $|Q|-r_Q-1 > 0$ for all $Q \subseteq[m]$ satisfying $|Q| \geq 2$. This allows us to state the following sparse version of Rado's Theorem.

\begin{theorem}[Rödl and Ruci\'nski~\cite{RR97}, Friedgut, Rödl and Schacht~\cite{FRS10}] \label{thm:sparsepartition}
	For every $r,m,s \in  \NN$ and partition regular matrix $A \in \mM_{r \times m}(\ZZ)$ there exist constants $c = c(A,s)$ and $C = C(A,s)$ such that
	\begin{equation*}
		\lim_{n \to \infty} \PP{[n]_{p} \to_s A} = \left\{\begin{array}{ll}
        0 & \quad \text{if } p(n) \leq c \, n^{-1/m_1(A)}, \\
        1 & \quad \text{if } p(n) \geq C \, n^{-1/m_1(A)}.
        \end{array}\right. 
	\end{equation*}
\end{theorem}
The current proof of this theorem is quite involved. A first goal of this note is to provide a short proof of the $1$-statement in Theorem~\ref{thm:sparsepartition} following the ideas of Nenadov and Steger's short proof of a sparse Ramsey Theorem~\cite{NS14}. This approach combines the recently developed hypergraph container framework by Balogh, Morris and Samotij~\cite{BMS14} as well as Saxton and Thomason~\cite{ST15} with a supersaturation result of Frankl, Graham and Rödl~\cite{FGR88}.

\medskip
	
Schacht~\cite{Sch12} as well as independently Conlon and Gowers~\cite{CG10} also stated a sparse version of Szémeredi's Theorem. Schacht also extended it to density regular systems as well as Schur triples. Given a set of integers $T$ and $\epsilon > 0$ we write
\begin{equation}
	T \to_{\epsilon} A	
\end{equation}
if every subset $S$ for which $|S|/|T| \geq \epsilon$ satisfies $S \cap \mS_0(A) \neq \emptyset$. An irredundant matrix $A \in \mM_{r \times m}(\ZZ)$ is \emph{density regular} if for all $\epsilon > 0$ we have $[n] \to_{\epsilon} A$ for $n$ large enough. Frankl, Graham and Rödl~\cite{FGR88} characterized density regular systems as those that are \emph{invariant}, that is they satisfy $A \cdot \mathbf{1} = \mathbf{0}$. Lastly, we say that a matrix $A \in \mM_{r \times m}(\ZZ)$ is \emph{positive} if $\mS(A) \cap \NN^m \neq \emptyset$ and \emph{abundant} if any $r \times (m-2)$ submatrix obtained from $A$ by deleting two columns has the same rank as $A$. Every density regular system is clearly partition regular and partition regular systems are irredundant and positive by definition. We will later see in Lemma~\ref{lemma:partitionabundant} that they are also abundant.

The second goal of this note is to extend Schacht's statement to the broadest group of matrices possible. To prove this generalization we derive a supersaturation result from a removal lemma due to Král', Serra and Vena~\cite{KSV12} and combine it with a corollary of the hypergraph containers due to Balogh, Morris and Samotij~\cite{BMS14}.

Given some matrix $A \in \mM_{r \times m}(\ZZ)$ let $\mathrm{ex}(n,A)$ be the size of the largest subset of $[n]$ not containing a proper solution and define $\pi (A) = \lim_{n \to \infty} \mathrm{ex}(n,A)/n$. Observe the clear parallels to the Turán number of a graph. Clearly density regular systems satisfy $\pi(A) = 0$ and for other systems systems we have $\pi(A) > 0$. One can easily bound this value away from $1$, as we will later see in Lemma~\ref{lemma:pibound}. These definitions and observations allow us to state the following sparse extremal result.

\begin{theorem} \label{thm:sparsedensity}
	For every $\epsilon > \pi(A)$, $r,m \in  \NN$ such that $m \geq 3$ and matrix $A \in \mM_{r \times m}(\ZZ)$ there exist constants $c = c(A,\epsilon)$ and $C = C(A,\epsilon)$ such that the following holds. If $A$ is irredundant, positive and abundant then 
	\begin{equation*}
		\lim_{n \to \infty} \PP{[n]_{p} \to_{\epsilon} A} = \left\{\begin{array}{ll}
        0 & \quad \text{if } p(n) \leq c \, n^{-1/m_1(A)}, \\
        1 & \quad \text{if } p(n) \geq C \, n^{-1/m_1(A)}.
        \end{array}\right. 
	\end{equation*}
\end{theorem}

\medskip

For singe-line equations it is common in combinatorial number theory to not just limit oneself to proper solutions, that is solutions with no repeated entries, but to also consider certain non-trivial solutions which may have some repeated entries. Ruzsa~\cite{Rz93} gave a definition for non-trivial solutions in this scenario and more recently Rué, S. and Zumalacárregui~\cite{RSZ15} extended it to include arbitrary homogeneous linear systems of equations. A third and final goal will therefore be to extend the previously stated sparse results to include non-trivial solutions. We need to introduce some notation in order to give a formal definition of this notion.

Given a solution $\bx = (x_1, \dots, x_m) \in \mS(A)$ let
\begin{equation}
	\fp[\bx] = \big\{ \{ 1 \leq j \leq m : x_i = x_j \} : 1 \leq i \leq m \big\}
\end{equation}
denote the set partition of the column indices $[m]$ indicating the repeated entries in $\bx$. Note that for $\bx \in \mS_0(A)$ we have $\fp[\bx] = \{ \{1\}, \dots, \{m\} \}$. Given some set partition $\fp$ of $\{1, \dots ,m\}$, let $A_{\fp}$ denote the matrix obtained by summing up the columns of $A$ according to $\fp$, that is for $\fp = \{ T_1, \dots, T_s \}$ such that $\min (T_1) < \dots < \min(T_s)$ for some $1 \leq s \leq m$ and $\bc_i$ the $i$-th column vector of $A$ for every $i \in [m]$, we have  
\begin{equation}
	A_{\fp} = \left( \begin{array}{ccccccc}
		\sum\limits_{i \in T_1} \! \bc_i & \Big| & \sum\limits_{i \in T_2} \! \bc_i & \Big| & \dots & \Big| & \sum\limits_{i \in T_s} \! \bc_i
 	\end{array} \right).
\end{equation}
A solution $\bx \in \mS(A)$ is now defined to be \emph{non-trivial} if $\rank{A_{\fp[\bx]}} = \rank{A}$. We denote the set of all non-trivial solutions by $\mS_1(A) = \{ \bx \in \mS(A) : \rank{A_{\fp[\bx]}} = \rank{A} \}$, that is we have $\mS(A) \supseteq \mS_1(A) \supseteq \mS_0(A)$. Lastly, given a set of integers $T$, an integer $s \in \NN$ and some $\epsilon > 0$, we write
\begin{equation}
	T \to^{\star}_s A
\end{equation}
if for every finite partition $T_1 \dot{\cup} \dots \dot{\cup} \, T_s = T$ there exists $1 \leq i \leq s$ such that $T_i \cap \mS_1(A) \neq \emptyset$ and
\begin{equation}
	T \to^{\star}_{\epsilon} A	
\end{equation}
if every subset $S$ for which $|S|/|T| \geq \epsilon$ also satisfies $S \cap \mS_1(A) \neq \emptyset$. This is just a direct extension of the previous notation to include non-trivial solutions. We now have the following two statements.
\begin{theorem} \label{thm:sparsepartitionnontrivial}
	For every $r,m,s \in \NN$ and partition regular matrix $A \in \mM_{r \times m}(\ZZ)$ there exists a constant $c = c(A,s)$ such that 
	\begin{equation*}
		\lim_{n \to \infty} \PP{[n]_{p} \to^{\star}_{s} A} = 0 \quad \text{if } p(n) \leq c \, n^{-1/m_1(A)}.
	\end{equation*}
\end{theorem}
\begin{theorem} \label{thm:sparsedensitynontrivial}
	For every $\epsilon > 0$, $r,m \in \NN$ such that $m \geq 3$ and irredundant, positive and abundant matrix $A \in \mM_{r \times m}(\ZZ)$ there exists a constant $c = c(A,\epsilon)$ such that 
	\begin{equation*}
		\lim_{n \to \infty} \PP{[n]_{p} \to^{\star}_{\epsilon} A} = 0 \quad \text{if } p(n) \leq c \, n^{-1/m_1(A)}.
	\end{equation*}
\end{theorem}
Observe that these results are stronger than the $0$-statements of Theorem~\ref{thm:sparsepartition} and Theorem~\ref{thm:sparsedensity} and that their respective $1$-statements supply matching counter-statements to Theorem~\ref{thm:sparsepartitionnontrivial} and Theorem~\ref{thm:sparsedensitynontrivial}. We will therefore only require proofs of the $1$-statements of Theorem~\ref{thm:sparsepartition} and Theorem~\ref{thm:sparsedensity} as well as the $0$-statements in Theorem~\ref{thm:sparsepartitionnontrivial} and Theorem~\ref{thm:sparsedensitynontrivial}.

\medskip

\subsection*{Outline.} In the remainder of this note we will first state some preliminaries in \emph{Section~\ref{sec:preliminaries}} about linear systems of equations, their subsystems and supersaturation results as well as introduce hypergraph containers. We will then proceed by providing short proofs based on hypergraph container results of the $1$-statements in Theorem~\ref{thm:sparsedensity} and Theorem~\ref{thm:sparsepartition} in \emph{Sections~\ref{sec:sparsedensity1statement}} and\emph{~\ref{sec:sparsepartition1statement}} respectively. The $0$-statements of Theorem~\ref{thm:sparsepartitionnontrivial} and Theorem~\ref{thm:sparsedensitynontrivial} will be proven in \emph{Sections~\ref{sec:sparsepartition0statement}} and\emph{~\ref{sec:sparsedensity0statement}} respectively.

\vspace{0.7cm}

\noindent \textbf{Acknowledgements.} I would like to thank Lluís Vena for his tremendous help and input regarding the removal lemma and the proof of the supersaturation results. I would also like to thank Juanjo Rué for his general assistance and supervision.

\section{Preliminaries} \label{sec:preliminaries}

Given some matrix $A \in \mM_{r \times m}(\ZZ)$, we have previously defined $\mathrm{ex}(n,A)$ to be the size of the largest subset of $[n]$ not containing a proper solution and $\pi (A) = \lim_{n \to \infty} \mathrm{ex}(n,A)/n$. Erd\H{o}s and Turan~\cite{ET41} determined that the size of the largest Sidon set, that is $A = (\begin{array}{cccc} 1 & 1 & -1 & -1 \end{array})$, satisfies $\mathrm{ex}(n,A) = \Theta(\sqrt{n})$. For sum-free subsets, that is $A = (\begin{array}{ccc} 1 & 1 & -1 \end{array})$, it is also easy to see that $\pi(A) = 1/2$. Hancock and Treglown~\cite{HT16} very recently extended this to matrices of the form $A = (\begin{array}{ccc} p & q & -r \end{array})$ where $p,q,r \in \NN$ such that $p \geq q \geq r$. Unfortunately, unlike the Erd\H{o}s-Stone-Simonovits Theorem~\cite{ES46,ES65} in the graph case, no exact characterization of $\pi(A)$ is known for arbitrary matrices $A$. However, the following lemma shows that one can still easily bound this value away from $1$ for every irredundant and positive matrix.

\begin{lemma}[Folklore] \label{lemma:pibound}
	Every irredundant and positive matrix $A \in \mM_{r \times m}(\ZZ)$ satisfies $\pi (A) < 1$.
\end{lemma}

\begin{proof}
Let $\bx = (x_1 , \dots , x_m) \in \mS_0(A) \cap \NN^m$. Clearly we also have $j \cdot \bx = (j x_1 , \dots , j x_m)  \in \mS_0(A) \cap \NN^m$ for any $j \geq 1$. Now for $n \geq m \max_i(x_i)$ we observe that every $i \in [n]$ can appear in at most $m$ of the $J = \floor{n/ \max_i(x_i)}$ solutions $\bx , 2 \cdot \bx , \dots , J \cdot \bx \in [n]^m$, so every subset of $[n]$ that avoids $\mS_0(A)$ is missing at least $J/m$ elements. It follows that $\pi (A) \leq (n - J/m)/n \leq 1-1/(m \max_i(x_i)) < 1$.	
\end{proof}

Partition and density regular matrices are irredundant and positive by definition. The next statement shows that Theorem~\ref{thm:sparsedensity} does indeed extend already existing results by proving that they are also abundant.

\begin{lemma} \label{lemma:partitionabundant}
If a given $A \in \mM_{r \times m}(\ZZ)$ is partition or density regular, then it is abundant.	
\end{lemma}

\begin{proof}
Rado characterized partition regular matrices as those that satisfy the column condition, that is it is possible to re-order the column vectors $\bc_1, \dots, \bc_m$ of $A$, so that for some choice of indices $0 = m_0 < m_1 < \dots < m_t = m$ setting $\bb_i = \sum_{j=m_{i-1}+1}^{m_i} \bc_j$ for $i = 1, \dots, t$ gives $\bb_1 = \bold{0}$ and $\bb_i$ can be expressed as a as a rational linear combination of $\bc_1, \dots, \bc_{m_{i-1}}$ for all $i \in \{  2, \dots, t \}$.

Assume now that $A$ is non-abundant, that is there exists a submatrix obtained by omitting two columns that has rank strictly smaller than $\rank{A}$. It follows that through basic row operations $A$ can be transformed into a matrix of full rank whose last row contains only two non-zero entries $a, b \in \ZZ \backslash \{ 0 \}$. As the matrix is partition regular, it is also irredundant and hence $a \neq -b$, that is $a + b \neq 0$. It follows that in order to satisfy the first requirement of the column condition, the columns need to be arranged such that there is a $0$ in the last entry of the first column. However, there now must exist some $2 \leq i \leq t$ such that the last entry in $\bb_i$ is non-zero while the last entries in $\bb_1, \dots, \bb_{i-1}$ are zero, violating the second requirement of the column condition. It follows that $A$ must have been abundant.
\end{proof}

Let us consider some examples to illustrate these categories. $A = (\begin{array}{ccc} 1 & 1 & -2 \end{array})$, that is the matrix associated with $3$-term arithmetic progression, is density regular by Roth's Theorem~\cite{Ro53}. It therefore is trivially also partition regular, which was previously established by van der Waerden~\cite{vdW27}. The matrix associated with $k$-term arithmetic progressions
\begin{equation}
	A = \left( \begin{array}{cccccc} 1 & -2 & 1 & & & \\ & 1 & -2 & 1 && \\ &&& \ddots && \\ &&& 1 & -2 & 1\end{array} \right) \in \mM_{(k-2) \times k}
\end{equation}
is density regular and therefore abundant by Szémeredi's Theorem~\cite{Sz75}. $A = (\begin{array}{ccc} 1 & 1 & -1 \end{array})$ is not density regular, but by Schur's Theorem it is still partition regular. Lastly, $A = (\begin{array}{ccc} 1 & 1 & -r \end{array})$ for $r \in \NN \backslash \{1,2\}$ is neither partition nor density regular but it is abundant. For some more examples see~\cite{RSZ15}.

\subsection{Counting Solutions} \label{subsec:countingsolutions}

We extend the notation from the introduction to inhomogeneous systems of linear equations. Given some matrix $A \in \mM_{r \times m}(\ZZ)$ and column vector $\bb \in \ZZ^r$ we write $\mS(A,\bb) = \{ \bx \in \ZZ^m : A \cdot \bx^T = \bb^T \}$, $\mS_0(A,\bb) = \{ \bx = (x_1, \dots , x_m) \in \mS(A,\bb) : x_i \neq x_j \text{ for } i \neq j \}$ and $\mS_1(A,\bb) = \{ \bx \in \mS(A,\bb) : \rank{A_{\fp[\bx]}} = \rank{A} \}$, so that $\mS(A) = \mS(A,\bold{0})$, $\mS_0(A) = \mS_0(A,\bold{0})$ and $\mS_1(A) = \mS_1(A,\bold{0})$. We remark that by elementary properties of systems of linear equations, we have the trivial upper bound 
	\begin{equation} \label{eq:trivialupperbound}
		\big| \, \mS_0(A,\bb) \cap [n]^m \big| \leq \big| \, \mS_1(A,\bb) \cap [n]^m \big| \leq \big| \, \mS(A,\bb) \cap [n]^m \big| \leq n^{m - \rank{A}}.
	\end{equation}
	The next lemma is due to Janson and Ruci\'nski~\cite{JR11} and establishes a lower bound that matches this up to a constant for homogeneous systems. It could be trivially extended to include non-homogeneous systems.

\begin{lemma}[Janson and Ruci\'nski~\cite{JR11}] \label{lemma:countingsolutions}
	Let $r,m \in \NN$ and a matrix $A \in \mM_{r \times m}(\ZZ)$ be given. If $\mS_0(A) \cap \NN^m$ is non-empty then there exists a constant $c_0 = c_0(A) > 0$ such that 
	\begin{equation}
		| \, \mS(A) \cap [n]^m| \geq | \, \mS_1(A) \cap [n]^m| \geq | \, \mS_0(A) \cap [n]^m| \geq c_0 \, n^{m-\rank{A}}.
	\end{equation}
\end{lemma}

In order to determine the exact asymptotic value of $| \, \mS(A) \cap [n]^m|/n^{m-\rank{A}}$ or $| \, \mS_0(A) \cap [n]^m|/n^{m-\rank{A}}$, one needs to employ Ehrhart's Theory, see for example Rué et al.~\cite{RSZ15}.

Lastly, let $\fP(A) = \{ \fp(\bx) : \bx \in \mS_1(A) \}$ denote the family of all set partitions of the column indices $[m]$ stemming from non-trivial solutions. The following lemma gives us the necessary tool to handle non-trivial solutions with repeated entries.
\begin{lemma} \label{lemma:nontrivial}
	For every $r,m \in \NN$, $A \in \mM_{r \times m}(\ZZ)$, partition $\fp \in \fP(A)$ and set $T \subset \NN$ we have
	\begin{equation}
		\big| \{ \bx \in \mS_1(A) \cap T^m : \fp[\bx] = \fp \} \big| \leq \big| \, \mS_0(A_{\fp}) \cap T^{|\fp|} \big|.
	\end{equation}
\end{lemma}

\begin{proof}
Write $\fp = \{ T_1, \dots, T_s\}$ for some $1 \leq s \leq m$ such that $\min(T_1) < \dots < \min(T_s)$. Let $Q = \{ \min(T_1), \dots, \min(T_s) \}$. Now for every $\bx = (x_1, \dots, x_m) \in \mS_1(A) \cap T^m$ such that $\fp[\bx] = \fp$, we would have $\bx^Q = (x_{\min(T_1)}, \dots, x_{\min(T_S)}) \in T^{|\fp|}$ as well as $\bx^Q \in \mS(A_{\fp})$ as can be readily seen by the definition of $A_{\fp}$. Since $\fp = \fp[\bx]$, the vector $\bx^Q$ would furthermore be proper, so that $\bx^Q \in \mS_0(A_{\fp}) \cap T^{|\fp|}$. The map $\{ \bx \in \mS_1(A) : \fp[\bx] = \fp \} \cap T^m \to \mS_0(A_{\fp}) \cap T^{|\fp|}, \: \bx \mapsto \bx^Q$ is clearly injective, proving the desired statement.
\end{proof}

An easy corollary of this result is clearly that if there are no proper solutions to the system $A_{\fp}$ in a set, then there can also not be non-trivial solutions to $A$ whose repetitions are indicated by $\fp$.

\subsection{Subsystems}

The notion of subsystems was originally introduced by Rödl and Ruci\'nski~\cite{RR97} when developing a sparse version of Rado's Partition Theorem. Recall the definitions from the introduction, especially $r_Q = \rank{A} - \rank{A^{\widebar{Q}}}$. Observe that we can without loss of generality assume that $A$ is of full rank for this part, since the solution space is unaffected by this assumption. This will simplify notation significantly.

For a given matrix $A \in \mM_{r \times m}(\ZZ)$ and column indices $\emptyset \subseteq Q \subseteq [m]$, we will now construct through basic row operations a matrix that tries to encapsulate the information contained in $A$ through the columns indexed by $Q$. Denote the rows of $A$ by $\ba_1, \ba_2, \dots, \ba_r$ so that the rows of $A^Q$ and $A^{\widebar{Q}}$ are respectively $\ba_1^Q, \ba_2^Q, \dots, \ba_r^Q$ and $\ba_1^{\widebar{Q}}, \ba_2^{\widebar{Q}}, \dots, \ba_r^{\widebar{Q}}$. Here we allow for empty vectors and matrices. If $\rank{A^{\widebar{Q}}} < \rank{A}$, then we can express exactly $r_Q > 0$ of the $r$ rows of $A^{\widebar{Q}}$ as linear combinations of the rest, that is there are indices $i_1 < \dots < i_{r_Q} \in [m]$ and integers $d_i, d_i^j \in \ZZ$ for $i \in \{ i_1, \dots, i_{r_Q} \}$ and $j \in [m] \backslash \{ i_1, \dots, i_{r_Q} \}$ so that
\begin{equation}
	d_i \, \ba_i^{\widebar{Q}} = \!\!\!\!\!\!\! \sum_{j \in [m] \backslash \{ i_1, \dots, i_{r_Q} \}} \!\!\!\!\!\!\! d_i^j \, \ba_j^{\widebar{Q}} \qquad \text{for } i \in \{ i_1, \dots, i_{r_Q} \}.
\end{equation}
Consider now the following integer-valued matrix with $r_Q$ rows and $|Q|$ columns
\begin{equation}
	\Sub{A}{Q} = \left( \begin{array}{c}
		d_{i_1} \, \ba_{i_1}^{Q} \: - \!\!\!\!\! \sum\limits_{j \in [m] \backslash \{ i_1, \dots, i_{r_Q} \}} \!\!\!\!\! d_{i_1}^j \, \ba_j^Q \\
		\vdots \\
		d_{i_{r_Q}} \ba_{i_{r_Q}}^{Q} \: - \!\!\!\!\! \sum\limits_{j \in [m] \backslash \{ i_1, \dots, i_{r_Q} \}} \!\!\!\!\! d_{i_{r_Q}}^j \, \ba_j^Q
	\end{array} \right) \in \mM_{r_Q \times |Q|}(\ZZ).
\end{equation}
To illustrate this construction further, note that if we assume that the column indices are appropriately ordered, that is $Q = \{1, \dots, |Q| \}$, then the matrix $A$ (without the assumption of being of full rank) can be rewritten as 
	\begin{equation} \label{eq:subsystemillustration}
		B = \left( \! \! \! \begin{array}{c}
				\text{------------} \\
				\begin{array}{cc} \Sub{A}{Q} & \mathbf{0} \textcolor{white}{\big|}\\ \mathbf{0} & \mathbf{0} \textcolor{white}{\big|} \end{array}
				\end{array} \! \! \! \right)
			\begin{array}{l}
				\big]\ \rank{A} - r_Q \textcolor{white}{\big|} \\
				\big]\ r_Q \textcolor{white}{\big|} \\
				\big]\ r - \rank{A} \textcolor{white}{\big|}
			\end{array}
	\end{equation}
	through elementary row operations, that is $A = P_1^{-1} \cdot B$ where $P_1 \in \mM_{r \times r}$ is an invertible rectangular matrix. We have $\mS(B) = \mS(A)$, that is the homogeneous solution space remains unchanged, at least up to the column permutation necessary to ensure that $Q = \{1, \dots, |Q| \}$. 
	
Observe that the matrix $A[Q]$ is only well defined up to our choices of indices $i_k$ and coefficients $d_i, d_i^j$. However, the homogeneous solution space $\mS(\Sub{A}{Q})$ is independent of these, so we pick one representative for each $\emptyset \subseteq Q \subseteq [m]$ and refer to it as the \emph{subsystem of $A$} induced by $Q$. The notation $\Sub{A}{Q}$ will refer to this particular representative. We state the following simple observations, that are immediately clear by considering Equation~\eqref{eq:subsystemillustration}.
	
\begin{remark} \label{rmk:subsystemobservations}
	$\Sub{A}{Q}$ is of full rank, that is $\rank{\Sub{A}{Q}} = r_Q$ for any $Q \subseteq [m]$ satisfying $r_Q > 0$. If $A$ was irredundant, positive or abundant, then $\Sub{A}{Q}$ trivially also fulfils these properties for any $Q \subseteq [m]$ such that $r_Q > 0$.
\end{remark}
	
The following lemma now establishes some results regarding the rank of subsystems of abundant matrices. It also verifies that the maximum $1$-density parameter given in the introduction is indeed well-defined for abundant matrices.

\begin{lemma}[Kusch et al.~\cite{KRSS16}] \label{lemma:welldefined}
	For any $r,m \in \NN$, abundant matrix $A \in \mM_{r \times m}(\ZZ)$ and selection of column indices $Q \subseteq [m]$ the following holds. If $|Q| \geq 2$ then we have $|Q| - r_Q - 1 > 0$, that is the parameter $m_1(A)$ is well-defined. If $|Q| = 1$ then we have $r_Q = 0$.
\end{lemma}

%\begin{proof}
%	First note that for any selection of column indices $Q$ we have $|Q| \geq r_Q $ as we have established that $\Sub{A}{Q}$ is of full rank. Next, assume that here exists a selection of column indices $Q$ satisfying $|Q| - r_Q = 1$ or $|Q| - r_Q = 0$, that is the subsystem $\Sub{A}{Q}$ has either one degree of freedom or none. In this case, $\Sub{A}{Q}$ would be equivalent to a matrix whose last line only contains only one or two non-zero entries. The same would hold for $\Sub{A}{\{ 1, \dots, m \}}$ which violates the requirement of abundancy for $A$.
%\end{proof}

The next lemma is crucial and establishes that a lack of non-trivial solutions to a subsystem of $A$ also implies a lack of non-trivial solutions to the full system. A proof of this as well as the previous statement can be found in Kusch et al.~\cite{KRSS16}. Note that this was previously proven by Rödl and Ruciński for proper solutions~\cite{RR97}.

\begin{lemma}[Kusch et al.~\cite{KRSS16}] \label{lemma:nosolutions}
	For any $r,m \in \NN$, matrix $A \in \mM_{r \times m}(\ZZ)$ and set $T \subset \NN$ the following holds. If there exists a selection of column indices $Q \subseteq [m]$ such that $r_Q > 0$ and $\mS_1(\Sub{A}{Q}) \cap T^{|Q|} = \emptyset$ then $\mS_1(A) \cap T^m = \emptyset$.
\end{lemma}

We end our observations about subsystem by stating the following easy proposition. It covers some trivial cases not considered by Theorem~\ref{thm:sparsedensitynontrivial} and will in fact be needed later in the proof of it.

\begin{proposition} \label{prop:trivialcases0statement}
	For every $\epsilon > 0$, $r,m \in \NN$ such that $m \geq 2$ and matrix $A \in \mM_{r \times m}(\ZZ)$ the following holds. If $A$ is irredundant, positive but not abundant, then we have $\lim_{n \to \infty} \PP{[n]_{p} \to^{\star}_{\epsilon} A} = 0$ for any $p(n) = o(1)$.
\end{proposition}

\begin{proof}
	Since $A$ is not abundant but positive and irredundant, there exists some $Q \subseteq [m]$ satisfying $|Q| = 2$ such that $\Sub{A}{Q} = (\begin{array}{cc} a & -b \end{array})$ for some $a,b \in \NN$, $a \neq b$. By Lemma~\ref{lemma:nosolutions} we can replace $A$ with $\Sub{A}{Q}$. It follows by Equation~\eqref{eq:trivialupperbound}, Lemma~\ref{lemma:countingsolutions} as well as the linearity of expectation that $\EE{|\mS(A) \cap [n]_p^2|} = \Theta(n \, p^2)$ while $\EE{|[n]_p|} = np$. If $np = O(1)$ then $\EE{|\mS(A) \cap [n]_p^2|} = o(1)$ and the result trivally holds by Markov's Inequality. If $np \to \infty$ then by Chernoff $|[n]_p| \geq np/2$ asymptotically almost surely. Since $p = o(1)$ we have $\EE{|\mS(A) \cap [n]_p^2|} = o(np/2)$ and therefore for any given set of positive density, we can remove one element per solution and still asymptotically almost surely have a solution-free set of that same density. This proves the desired result.
\end{proof}

\subsection{Removal Lemma and Supersaturation Results}

A common ingredient to proving sparse results are robust versions of the deterministic statement, referred to as supersaturation results. In the graph setting such a result is folklore and easy to prove. A number theoretical counterpart is Varnavides~\cite{Va59} robust version of Szemerédi's Theorem which states that a set of positive density contains not just one, but a positive proportion of all $k$–term arithmetic progressions. Frankl, Graham and Rödl~\cite{FGR88} formulated such results for both for partition and density regular systems.

\begin{lemma}[Theorem~1 in Frankl, Graham and Rödl~\cite{FGR88}] \label{lemma:partitionsupersaturation}
	For a given partition regular matrix $A \in \mM_{r \times m}(\ZZ)$ and $s \in \NN$ there exists $\zeta  = \zeta(A,s) > 0$ such that for any partition $[n] = T_1 \dot{\cup} \dots \dot{\cup} \, T_s$ and $n$ large enough we have $|\mS_0(A) \cap T_1^m| + \dots + |\mS_0(A) \cap T_s^m| \geq \zeta \, | \, \mS_0(A) \cap [n]^m|$.
\end{lemma}

\begin{lemma}[Theorem~2 in Frankl, Graham and Rödl~\cite{FGR88}] \label{lemma:densitysupersaturation}
	For a given density regular matrix $A \in \mM_{r \times m}(\ZZ)$ and $\delta > 0$ there exists $\zeta  = \zeta(A,\delta) > 0$ such that any subset $T \subseteq [n]$ satisfying $|T| \geq \delta n$ contains at least $\zeta \, | \, \mS_0(A) \cap [n]^m|$ proper solutions for $n$ large enough.
\end{lemma}

We will extend Lemma~\ref{lemma:densitysupersaturation} to cover the scope of this note by using an \emph{arithmetic removal lemma}. Green~\cite{Gr05} first formulated such a statement for linear equations in an abelian group. Later Shapira~\cite{Sh10} as well as independently Král', Serra and Vena~\cite{KSV12} proved a removal lemma for linear maps in finite fields. We will state it here in a simplified version.

\begin{theorem}[Removal Lemma~\cite{KSV12}] \label{thm:removallemma}
	Let $\FF_q$ be the finite field of order $q$. Let $X \subset \FF_q$ be a subset of $\FF_q$ and $A \in \mM_{r \times m}(\FF_q)$ a matrix of full rank. For $\mS = \{ \bx \in \FF_q^m : A \cdot \bx^T = \mathbf{0}^T \}$ and every $\epsilon > 0$ there exists an $\eta = \eta(\epsilon,r,m)$ such that if $|\mS \cap X^m| < \eta \, |\mS|$ then there exists a set $X' \subset X$ with $|X'| < \epsilon q$ and $\mS \cap (X \backslash X')^m = \emptyset$.
\end{theorem}

Applying this result, we formulate the following extension of Lemma~\ref{lemma:densitysupersaturation}.

\begin{lemma}[Supersaturation] \label{lemma:extendeddensitysupersaturation}
	For a given $r,m \in \NN$, positive and irredundant matrix $A \in \mM_{r \times m}(\ZZ)$ and $\delta > \pi (A)$ there exists $\zeta  = \zeta(\delta,A) > 0$ such that any subset $T \subseteq [n]$ satisfying $|T| \geq \delta n$ contains at least $\zeta \, | \, \mS_0(A) \cap [n]^m|$ proper solutions for $n$ large enough.
\end{lemma}

\begin{proof}
	Let $q = q(A,n)$ be a prime number between $2 m n \max(|A|)$ and $4 m n \max(|A|)$ and $\FF_q$ the finite field with $q$ elements. Here $\max (|A|)$ refers to the maximal absolute entry in $A$. Note that such a prime number exists for example because of the Bertrand–Chebyshev Theorem. We have $\FF_q \cong \ZZ_q$ and we can identify the integers with their corresponding residue classes in $\FF_q$. The matrix $A$ now defines a map from $\FF_q^m$ to $\FF_q^r$. A solution in $\mS(A)$ clearly lies in the $\mS$ and, as we have chosen $q$ large enough, all canonical representatives from $\mS \cap [n]^m$ also lie in $\mS(A) \cap [n]^m$ for $n \geq \max |A|$.

	Next, set $\delta' = (\delta + \pi (A))/2$ and let $n$ be large enough such that any subset of density at least $\delta'$ in  $[n]$ contains a proper solutions. Note that $\delta > \delta' > \pi(A)$. Given a subset $T \subseteq [n]$ satisfying $|T| \geq \delta n$ consider the corresponding set $X$ of residue classes in $\FF_q$. One needs to remove at least $(\delta - \delta') n$ elements from $T$ in order for $T^m$ to avoid $\mS_0(A)$ in $[n]$, so one needs to remove at least an
		$$\epsilon = \frac{(\delta - \delta') n}{q} \geq \frac{(\delta - \delta')}{4 m \max(|A|)} > 0$$
		proportion of elements in $\FF_q$ from $X$ so that $X^m$ avoids $\mS$ in $\FF_q$. It follows from Theorem~\ref{thm:removallemma} that $|\mS \cap X^m| \geq \eta |\mS|$ for some $\eta = \eta(\epsilon,r,m)$. Since we have chosen $q$ large enough, it follows that $T$ contains at least an $\eta$ proportion of $\mS(A) \cap [n]^m$. An easy consequence of Equation~\eqref{eq:trivialupperbound} and Lemma~\ref{lemma:countingsolutions} is that $\lim_{n \to \infty} | \, \mS_0(A) \cap [n]^m| / | \, \mS(A) \cap [n]^m| \geq c_0$ for $c_0 = c_0(A) > 0$ as given by Lemma~\ref{lemma:countingsolutions}. It follows result holds for $n$ large enough and $\zeta = \zeta(\delta,A) = (c_0 \, \eta) / 2$.
\end{proof}

\subsection{Hypergraph Containers}

The development of \emph{hypergraph containers} by Balogh, Morris and Samotij~\cite{BMS14} as well as independently Thomason and Saxton~\cite{ST15} has opened a new, easy and unified framework to proving sparse results. Let us start by stating the Hypergraph Container Theorem as given by Balogh, Morris and Samotij.

Given a hypergraph $\mH$ we denote its vertex set by $V(\mH)$ and its set of hyperedges by $E(\mH)$. The cardinality of these sets will be respectively denoted by $v(\mH)$ and $e(\mH)$. Given some subset of vertices $A \subseteq V(\mH)$ we denote the subgraph it induces in $\mH$  by $\mH[A]$ and its degree by $\deg_{\mH}(A) = |\{ e \in E(\mH) : A \subseteq e \}|$. For $\ell \in \NN$ we denote the \emph{maximum $\ell$-degree} by $\Delta_{\ell}(\mH) = \max \{ \deg_{\mH}(A) : A \subseteq V(\mH) \text{ and } |A| = \ell \}$. Let the set of independent vertex sets in $\mH$ be denoted by $\mI(\mH)$. Lastly, let $\mH$ be a uniform hypergraph, $\mF$ an increasing family of subsets of $V(\mH)$ and $\epsilon > 0$. We say that $\mH$ is \emph{$(\mF,\epsilon)$-dense} if $e(\mH[A]) \geq \epsilon \, e(\mH)$ for every $A \in \mF$.

\begin{theorem}[Hypergraph Containers, Theorem~2.2 in~\cite{BMS14}] \label{thm:hypergraphcontainer}
	For every $m \in \NN$, $c > 0$ and $\epsilon > 0$, there exists a constant $C = C(m,c,\epsilon) > 0$ such that the following holds. Let $\mH$ be an $m$-uniform hypergraph and let $\mF \subseteq 2^{V(\mH)}$ be an increasing family of sets such that $|A| \geq \epsilon v(\mH)$ for all $A \in \mF$. Suppose that $\mH$ is $(\mF,\epsilon)$-dense and $p \in (0,1)$ is such that, for every $\ell \in\{1, \dots, k \}$,
		\begin{equation}
			\Delta_{\ell}(\mH) \leq c \; p^{\ell-1} \, \frac{e(\mH)}{v(\mH)}.
		\end{equation}
		Then there exists a family $\mT \subseteq \binom{V(\mH)}{\leq Cp \, v(\mH)}$ and functions $f : \mT \to \widebar{\mF}$ and $g : \mI(\mH) \to \mT$ such that for every $I \in \mI(\mH)$,
		\begin{equation}
			g(I) \subseteq I \quad \text{and} \quad I \backslash g(I) \subseteq f(g(I)).
		\end{equation}
\end{theorem}

The statement gives the existence of a small number of \emph{containers} $\widebar{\mF}$ and some \emph{fingerprints} $\mT$ so that every independent set $I$ in $\mH$ is identified with a fingerprint $g(I)$ that determines a container $f(g(I))$ which contains the independent set.

\medskip

Next, let $\mH = (\mH_n)_{n \in \NN}$ be a sequence of $m$-uniform hypergraphs and let $\alpha \in [0,1)$. We say that $\mH$ is \emph{$\alpha$-dense} if for every $\delta > 0$, there exist some $\epsilon > 0$ such that for $U \subseteq V(\mH_n)$ which satisfies $|U| > (\alpha + \delta) \, v(\mH_n)$ we have $e(\mH_n[U]) > \epsilon \, e(\mH_n)$ for $n$ large enough. Balogh, Morris and Samotij proved the following consequence of their container statement.

\begin{theorem}[Sparse Sets through Hypergraph Containers, Theorem~5.2 in~\cite{BMS14}] \label{thm:sparsehypergraphcontainer}
	Let $\mH = (\mH_n)_{n \in \NN}$ be a sequence of $m$-uniform hypergraphs, $\alpha \in [0,1)$ and let $C > 0$. Suppose that $q = q(n)$ is a sequence of probabilities such that for all sufficiently large $n$ and for every $\ell \in\{ 1, \dots , m \}$ we have
	\begin{equation}
		\Delta_{\ell}(\mH_n) \leq C \, q(n)^{\ell-1} \, \frac{e(\mH_n)}{v(\mH_n)}.
	\end{equation}
	If $\mH$ is $\alpha$-dense, then for every $\delta > 0$, there exists a constant $c = c(C,\alpha,m) > 0$ such that if $p(n) > c \, q(n)$ and $p(n) \, v(\mH_n) \to \infty$ as $n \to \infty$, then asymptotically almost surely 
	\begin{equation}
		\alpha \big( \mH_n [V(\mH_n)_{p(n)}] \big) \leq ( \alpha + \delta ) \, p(n) \, v(\mH_n).
	\end{equation}
\end{theorem}

We will make use of this statement in order to obtain a proof for the $1$-statement of Theorem~\ref{thm:sparsedensity}. For a proof of the $1$-statement of Theorem~\ref{thm:sparsepartition} such a ready-made statement does not exist and we will follow Nenadov and Steger's~\cite{NS14} short proof of a sparse Ramsey statement by applying Theorem~\ref{thm:hypergraphcontainer}.

%\subsection{Probability}
%
%Given a random variable $X$ we denote its expectation as $\EE{X}$ and its variance by $\Var{X}$. \emph{Markov's Inequality} states that for $X$ satisfying $\PP{X \geq 0} = 1$ and $t > 0$ we have
%%
%\begin{equation} \label{eq:markov}
%	\PP{X \geq t} \leq \frac{\EE{X}}{t}.
%\end{equation}
%%
%A straight forward consequence of Markov's Inequality is that $\lim_{n \to \infty} \PP{X_n > 0} = 0$ for a sequence of random variable $X = (X_n)_{n \in \NN}$ satisfyings $\lim_{n \to \infty} \EE{X_n} = 0$. \emph{Chebyshev's inequality} states that for a random variable $X$ whose variance exists and $t > 0$ we have
%%
%\begin{equation} \label{eq:chebyshev}
%	\PP{|X - \EE{X}| \geq t} \leq \frac{\Var{X}}{t^2}.
%\end{equation}
%%
%There are many results know as \emph{Chernoff bounds} that give exponential bounds on the tail distribution of sums of independent random variables. We will just use the simple consequence that 
%%
%\begin{equation} \label{eq:chernoff}
%	\lim_{n \to \infty} \PP{\mathcal{B}(n,p) < \frac{np}{2}} = 0.
%\end{equation}
%%
%for the binomial distribution $\mathcal{B}(n,p)$ with parameter $p = p(n)$ such that $np \to \infty$.

\section{Proof of the 1-statement in Theorem~\ref{thm:sparsedensity}} \label{sec:sparsedensity1statement}

Let $\mH_n$ be the hypergraph with vertex set $V(\mH_n) = [n]$ and edge multiset
$$E(\mH_n) = \big\{ \! \big\{ \{ x_1 , \dots , x_m \} : (x_1, \dots , x_m) \in \mS_0(A) \cap [n]^m \big\} \! \big\}.$$
Observe that $H_n$ can be a multigraph, that is multiple edges are allowed, but the multiplicity of each edge is clearly bounded by $m!$. We do this to simplify counting, since this way we have $|E(\mH_n)| = | \, \mS_0(A) \cap [n]^m|$. We observe that we can limit ourselves to proper solutions when proving the $1$-statement.

Corollary~\ref{lemma:extendeddensitysupersaturation} now states that $\mH = (\mH_n)_{n \in \NN}$ is $\pi(A)$-dense. In order to apply Theorem~\ref{thm:sparsehypergraphcontainer}, it remains to determine a sequence $q = q(n)$ satisfying the required condition. The following lemma gives us upper bounds for the maximum $\ell$-degrees in $\mH_n$.

\begin{lemma} \label{lemma:degreeupperbound}
	For $1 \leq \ell \leq m$ we have $\Delta_{\ell} (\mH_n) \leq \ell! \, m^{\ell} \max_{Q \subseteq [m], \; |Q| = \ell} n^{(m-\rank{A})-(|Q|-r_Q)}$.
\end{lemma}

\begin{proof}
For $\mH = (\mH_n)$ as defined above and $\ell \in\{ 1, \dots, m \}$ we have 
\begin{align*}
	\Delta_{\ell} (\mH_n) & \leq \max_{x_1 , \dots , x_{\ell} \in  [n]} \big| \{ \bx \in \mS_0(A) \cap [n]^m : \exists Q \subseteq [m], \pi \in S(\ell) \text{ s.t. } \bx^Q = (x_{\pi(1)}, \dots , x_{\pi(\ell)}) \} \big| \\
	& \leq \ell! \, \binom{m}{\ell} \max_{\substack{(x_1 , \dots , x_{\ell}) \in  [n]^{\ell} \\ Q \subseteq [m] , |Q| = \ell}} \big| \{ \bx \in [n]^{m-l} \mid A^{\widebar{Q}} \cdot \bx^T = -A^{Q} \cdot (x_1, \dots, x_{\ell})^T \} \big| \\
	& \leq \ell! \, m^{\ell} \max_{\substack{Q \subseteq [m] \\ |Q| = \ell}} \, \max_{\bb \in \ZZ^r} \, \big| \, \mS(A^{\widebar{Q}},\bb) \cap [n]^m \big| \leq \ell! \, m^{\ell} \max_{\substack{Q \subseteq [m] \\ |Q| = l}} n^{|\widebar{Q}| - \rank{A^{\widebar{Q}}}} \\
	& = \ell! \, m^{\ell} \max_{\substack{Q \subseteq [m] \\ |Q| = l}} n^{(m-\rank{A})-(|Q|-r_Q)}
\end{align*}
where $S(\ell)$ denotes the set of permutations of $\ell$ elements. We have also made extensive use of the notation defined in the introduction as well as as the trivial upper bound for the number of solutions stated in Equation~\eqref{eq:trivialupperbound}.
\end{proof}
	
Note that $r_Q = 0$ for any $Q \subseteq [m]$ satisfying $|Q| = 1$ due to Lemma~\ref{lemma:welldefined} and that there exists $c_0 = c_0(A) > 0$ such that $e(\mH_n) \geq c_0 \, n^{m-\rank{A}}$ due to Lemma~\ref{lemma:countingsolutions}. Using Lemma~\ref{lemma:degreeupperbound} we now observe that
\begin{align*}
	\Delta_1(\mH_n) & \leq m \, n^{m-\rank{A}-1} \leq m/c_0 \; \frac{e(\mH_n)}{v(\mH_n)}.
\end{align*}
For $\ell \in\{ 2, \dots , m \}$ we again apply Lemma~\ref{lemma:degreeupperbound} to see that 
\begin{align*}
	\Delta_{\ell}(\mH_n) & \leq \ell! \, m^{\ell} \max_{Q \subseteq [m], \; |Q| = \ell} n^{(m-\rank{A})-(|Q|-r_Q)} = \ell! \, m^{\ell} \, \Big( \max_{Q \subseteq [m], \; |Q| = \ell} n^{-\frac{|Q|-r_Q-1}{|Q|-1}} \Big)^{\ell-1} \; n^{m-\rank{A}-1} \\
	& \leq \ell! \, m^{\ell} \, \big( n^{-1/m_1(A)} \big)^{\ell-1} \; n^{m-\rank{A}-1} \leq (\ell! \, m^{\ell})/c_0 \, \big( n^{-1/m_1(A)} \big)^{\ell-1} \: \frac{e(\mH_n)}{v(\mH_n)}.
\end{align*}
Lastly we observe that $n^{-1/m_1(A)} \, v(\mH_n) = n^{1-1/m_1(A)} \to \infty$ as $m_1(A) > 1$. It follows that the prerequisites of Theorem~\ref{thm:sparsehypergraphcontainer} hold for $C = (m! \, m^m)/c_0$, $q = q(n) = n^{-1/m_1(A)}$ and we can choose the $c = c(A,\epsilon)$ in Theorem~\ref{thm:sparsedensity} to be equal to the $c = c(C,\pi(A),m)$ as given by Theorem~\ref{thm:sparsehypergraphcontainer}.

\section{A Short Proof of the 1-statement in Theorem~\ref{thm:sparsepartition}} \label{sec:sparsepartition1statement}

As stated in the introduction, this result was previously proven by Friedgut, Rödl and Schacht~\cite{FRS10} as well as independently Conlon and Gowers~\cite{CG10}. This prove merely serves as a short version that follows the short proof of a sparse Ramsey result due to Nenadov and Steger~\cite{NS14}.

\medskip

We will need two ingredients in order to prove the $1$-statement of Theorem~\ref{thm:sparsepartition}. The first will be the following easy corollary to Lemma~\ref{lemma:partitionsupersaturation}.

\begin{corollary} \label{cor:partitionsupersaturationcorollary}
	For a given partition regular matrix $A \in \mM_{r \times m}(\ZZ)$ and $s \in \NN$ there exist $\epsilon = \epsilon(A,s)$ and $\delta = \delta(A,s) > 0$ such that for any $T_1, \dots, T_s \subseteq [n]$ satisfying $|\mS_0(A) \cap T_i^m| \leq \epsilon \, |\mS_0(A) \cap [n]^m|$ for $1 \leq i \leq s$ we have $\big| [n] \backslash (T_1 \cup \dots \cup T_s) \big| \geq \delta n$ for $n$ large enough.
\end{corollary}

\begin{proof}
	Let $\zeta = \zeta(A,s+1)$ be as in Lemma~\ref{lemma:partitionsupersaturation} and $\epsilon = \epsilon(A,s) = \zeta/2s$. Set $\tilde{T}_i = T_i \backslash \bigcup_{j=1}^{i-1} T_j$ for $1 \leq i \leq s$ and $\tilde{T}_{s+1} = [n] \backslash \bigcup_{j=1}^{r} T_j$ and consider the partition $[n] = \tilde{T}_1 \, \dot{\cup} \, \dots \, \dot{\cup} \, \tilde{T}_s \, \dot{\cup} \, \tilde{T}_{s+1}$. By Lemma~\ref{lemma:partitionsupersaturation} we have $|\mS_0(A) \cap \tilde{T}_1^m| + \dots + |\mS_0(A) \cap \tilde{T}_{r+1}^m| \geq \zeta \, |\mS_0(A) \cap [n]^m|$ and since by assumption $|\mS_0(A) \cap \tilde{T}_i^m| \leq |\mS_0(A) \cap T_i^m| \leq \zeta/2s \, |\mS_0(A) \cap [n]^m|$ for all $i \in \{1, \dots, s \}$, we have $|\mS_0(A) \cap \left( [n] \backslash (T_1 \cup \dots \cup T_s) \right)^m| \geq \zeta/2 \, |\mS_0(A) \cap [n]^m|$.  Observe that by Lemma~\ref{lemma:degreeupperbound} every element in $[n]$ is contained in at most $m \, n^{m-\rank{A}-1}$ solutions and by Lemma~\ref{lemma:countingsolutions} there exists $c_0 = c_0(A) > 0$ such that $| \, \mS_0(A) \cap [n]^m | \geq c_0 \, n^{m-\rank{A}}$ for $n$ large enough, so that the results follows for $\delta = \zeta c_0/2$.
\end{proof}

The second ingredient is stated in the following corollary that is obtained by applying the Hypergraph Container Theorem to the hyperpgraph of solutions. 

\begin{corollary} \label{cor:containercorollary}
	For a given partition regular matrix $A \in \mM_{r \times m}(\ZZ)$ and $\epsilon > 0$ there exist $t = t(n)$ sets $T_1, \dots , T_t \in \binom{[n]}{\leq c_0 \, n^{1-1/m_1(A)}}$ for some $c_0 > 0$ as well as sets $C_1, \dots, C_t \subseteq [n]$ such that
	\begin{equation}
		|\mS_0(A) \cap C_i^m| \leq \epsilon \, |\mS_0(A) \cap [n]^m|.
	\end{equation}
	Furthermore, for every set $T \subseteq [n]$ satisfying $\mS_0(A) \cap T^m = \emptyset$ there exists $1 \leq i \leq t$ such that 
	\begin{equation}
		T_i \subseteq T \subseteq C_i.
	\end{equation}
\end{corollary}

\begin{proof}
	Let $\mH_n$ again be the hypergraph with vertex set $V(\mH_n) = [n]$ and edge multiset
	$$E(\mH_n) = \big\{ \! \big\{ \{ x_1 , \dots , x_m \} : (x_1, \dots , x_m) \in \mS_0(A) \cap [n]^m \big\} \! \big\}.$$
	We have previously observed that there exists a $c > 0$ such that $\Delta_{\ell} \leq c \, p(n)^{\ell-1} \, e(\mH_n)/v(\mH_n)$ for $p = p(n) = n^{-1/m_1(A)}$. We observe that $\mH_n$ is trivially $(\mF,\epsilon)$-dense for $\mF = \{ T \subseteq [n] : |\, \mS_0(A) \cap T^m| \geq \epsilon \, |\, \mS_0(A) \cap [n]^m| \}$. Applying Theorem~\ref{thm:hypergraphcontainer} gives the desired statement.
\end{proof}

We are now ready to give a short proof of the $1$-statement in Theorem~\ref{thm:sparsepartition} following the ideas of Nenadov and Steger~\cite{NS14}. Let $\epsilon, \delta > 0$ be as in Corollary~\ref{cor:partitionsupersaturationcorollary} and let $t = t(n)$, $c_0$, $S_1, \dots, S_t$ and $C_1, \dots, C_t$ be as in Corollary~\ref{cor:containercorollary}. Let $C = C(A,s) \geq 2 s c_0 / \delta$ be constant.

Observe now that for a partition of the random set $T_1 \dot{\cup} \dots \dot{\cup} \, T_s = [n]_p$ satisfying $\mS_0(A) \cap T_i^m = \emptyset$ for all $i \in \{1, \dots, s\}$ there exist $j_1, \dots, j_s \in \{1, \dots, t\}$ so that $S_{j_i} \subseteq T_i \subseteq C_{j_i}$ for all $i \in \{1, \dots, s\}$. Since $T_i \subseteq [n]_p$ for $1 \leq i \leq s$ and $[n] \backslash (C_1 \cup \dots \cup \, C_s) \cap [n]_p = \emptyset$ we can bound the probability of $[n]_p$ not fulfilling the partition property by
\begin{align*}
\PP{[n]_p \not\to_s A} \leq \! \! \! \! \! \! \! \sum_{j_1, \dots, j_s \in \{1, \dots, t\}}  \! \! \! \! \! \! \! \PP{S_{j_1}, \dots, S_{j_s} \subseteq [n]_p \: \wedge \: [n] \backslash (C_{j_1}, \dots, C_{j_s}) \cap [n]_p = \emptyset }.
\end{align*}
Observe that the two events $S_{j_1}, \dots, S_{j_s} \subseteq [n]_p$ and $[n] \backslash (C_{j_1}, \dots, C_{j_s}) \cap [n]_p = \emptyset$ are independent, so that we have
\begin{align*}
\PP{[n]_p \not\to_s A} \leq  \! \! \! \! \! \! \! \sum_{j_1, \dots, j_s \in \{1, \dots, t\}}  \! \! \! \! \! \! \! p^{\, \left| \bigcup_{j=1}^s S_j \right|} \, (1-p)^{\, \left| [n] \backslash (C_{j_1}, \dots, C_{j_s}) \right|}.
\end{align*}
We bound this by choosing $k = | \bigcup_{j=1}^s S_j | \leq s c_0 n^{1-1/m_1(A)}$, then picking $k$ elements and lastly deciding for each element in this selection in which of the $S_i$ it is contained, so that we have 
\begin{align*}
\PP{[n]_p \not\to_s A} \leq (1-p)^{\delta n} \, \! \! \sum_{k = 0}^{s c_0 n^{1-1/m_1(A)}} \! \!\binom{n}{k} \left( 2^{s} \right)^k p^k & \leq e^{-\delta n p} \, \! \! \sum_{k = 0}^{s c_0 C^{-1} \, np} \! \! \left( \frac{e 2^s \, np}{k}  \right)^k.
\end{align*}
Lastly we note that for $c > 0$ the function $f(x) = (c/x)^x$ is increasing for $0 \leq x \leq c/e$ since $d/dx \, f(x) = (c/x)^x \left( \log(c/x) - 1 \right)$. We have chosen $C$ large enough so that for $n$ large enough we have 
\begin{align*}
\PP{[n]_p \not\to_s A} \leq e^{-\delta n p} \, (s c_0 C^{-1} \, np + 1) \left( \frac{e 2^{s c_0 C^{-1} \, np}}{s c_0 C^{-1}}  \right)^k \leq e^{-\delta n p} \; e^{\delta n p/2} = o(1).
\end{align*}
As desired it follows that $[n]_p \to_s A$ asymptotically almost surely for $p \geq C \, n^{-1/m_1(A)}$.

\section{Proof of Theorem~\ref{thm:sparsepartitionnontrivial} -- A Rado-type 0-statement} \label{sec:sparsepartition0statement}

Let $Q \subseteq [m]$ be a set of column indices satisfying $|Q| \geq 2$ such that $(|Q|-1)/(|Q|-r_{Q}-1) = m_1(A)$. Due to Lemma~\ref{lemma:nosolutions} we can replace $A$ with $\Sub{A}{Q}$ if necessary in order to guarantee that $(m-1)/(m-\rank{A}-1) = m_1(A)$. Due to Lemma~\ref{lemma:nontrivial} we know that
\begin{equation} \label{eq:splittinguppartition}
	\PP{[n]_p \to_s^{\star} A} \leq \PP{ \, \bigcup_{\fp \in \fP(A)} \Big( [n]_p \to_s A_{\fp} \Big) } \leq \sum_{\fp \in \fP(A)} \PP{ [n]_p \to_s A_{\fp} }.
\end{equation}
Let us bound the individual probabilities $\PP{ [n]_p \to_s A_{\fp} }$ for each $\fp \in \fP(A)$. For $|\fp| = m$, that is $\fp = \{\{1\}, \dots, \{m\}\}$, we know due to Rödl and Ruciński's Theorem~\ref{thm:sparsepartition} that there exists a $c = c(A,s)$ such that $\lim_{n \to \infty} \PP{[n]_p \to_s A} = 0$ for $p = p(n) \leq c \, n^{-1/m_1(A)}$. For $|\fp| < m$ we consider two separate cases. If $A_{\fp}$ is not partition regular, then $[n] \not\to_s A$ and therefore trivially $\lim_{n \to \infty} \PP{ [n]_p \to_s A_{\fp} } = 0$. If $A_{\fp}$ is partition regular, then 
\begin{equation*}
	m_1(A_{\fp}) \geq \frac{|\fp|-1}{|\fp|-\rank{A}-1} > \frac{m-1}{m-\rank{A}-1} = m_1(A)	
\end{equation*}
so that $n^{-1/m_1(A)} = o \big( n^{-1/m_1(A_{\fp})} \big)$ and therefore again by Theorem~\ref{thm:sparsepartition} we have $\lim_{n \to \infty} \PP{[n]_p \to_s A_{\fp}} = 0$ for $p = p(n) \leq c \, n^{-1/m_1(A)}$. The desired statement follows due to Equation~\eqref{eq:splittinguppartition}.

\section{Proof of Theorem~\ref{thm:sparsedensitynontrivial} -- A Szémeredi-type 0-statement} \label{sec:sparsedensity0statement}

Due to Lemma~\ref{lemma:nontrivial} we know that
\begin{equation} \label{eq:splittingupdensity}
	\PP{[n]_p \to_{\epsilon}^{\star} A} \leq \PP{ \, \bigcup_{\fp \in \fP(A)} \Big( [n]_p \to_{\epsilon} A_{\fp} \Big) } \leq \sum_{\fp \in \fP(A)} \PP{ [n]_p \to_{\epsilon} A_{\fp} }.
\end{equation}
We will therefore analyze the individual probabilities $\PP{ [n]_p \to_{\epsilon} A_{\fp} }$ for each $\fp \in \fP(A)$. The constant $c = c(A,\epsilon)$ will be define later in Equation~\eqref{eq:smallc}. We start by first stating the following proposition, which restricts the statement of Theorem~\ref{thm:sparsedensitynontrivial} to proper solutions. Its proof will be given at the end of this section.
\begin{proposition} \label{prop:densityproper0}
	For every $\epsilon > 0$, $r,m \in \NN$ such that $m \geq 3$ and matrix $A \in \mM_{r \times m}(\ZZ)$ there exists a constant $c = c(A,\epsilon)$ such that the following holds. If $A$ is irredundant, positive and abundant, then we have $\lim_{n \to \infty} \PP{[n]_{p} \to_{\epsilon} A} = 0$ if $p(n) \leq c \, n^{-1/m_1(A)}$.
\end{proposition}

For $|\fp| < m$ we now observe that $A_{\fp}$ again clearly is irredundant and positive since $\fp$ indicates the repeated entries of an actual solution in $\mS_1(A)$. If $A_{\fp}$ is not abundant, then Proposition~\ref{prop:trivialcases0statement} states that $\lim_{n \to \infty} \PP{ [n]_p \to_{\epsilon} A_{\fp} } = 0$ for $p = p(n) \leq c \, n^{-1/m_1(A)} = o(1)$ independent of the constant $c$. If $A_{\fp}$ is abundant, then we can apply Proposition~\ref{prop:densityproper0} to it. If we assume as in the proof of Theorem~\ref{thm:sparsepartitionnontrivial} that $m_1(A) = (m-1)/(m-\rank{A}-1)$, then we again have $n^{-1/m_1(A)} = o \big( n^{-1/m_1(A_{\fp})} \big)$ and therefore $\lim_{n \to \infty} \PP{[n]_p \to_s A_{\fp}} = 0$ for $p = p(n) \leq c \, n^{-1/m_1(A)}$ independent of $c$. Lastly, let $|\fp| = m$, that is $\fp = \{\{1\},\dots,\{m\}\}$ and therefore $A_{\fp} = A$. Proposition~\ref{prop:densityproper0} applies to $A$ and therefore we obtain the desired statement with $c = c(A,\epsilon)$ as given by Proposition~\ref{prop:densityproper0}. The desired statement now follows due to Equation~\eqref{eq:splittingupdensity}.

\medskip

\begin{proof}[Proof of Proposition~\ref{prop:densityproper0}]
	Observe that the expected number of elements in $[n]_p$ is
\begin{equation} \label{eq:nrofelements}
	\EE{|[n]_p|} = np.
\end{equation}
We also note that due to Lemma~\ref{lemma:countingsolutions} there exists $c_0 = c_0(\Sub{A}{Q})$ such that 
\begin{align} \label{eq:nrofsolutions_lowerbound}
	\EE{\big| \, \mS_0(\Sub{A}{Q}) \cap [n]_p^{|Q|} \big|} \geq c_0 \, n^{|Q| - r_Q} \, p^{|Q|} \quad \text{for} \quad \emptyset \neq Q \subseteq [m]
\end{align}
and due to Equation~\eqref{eq:trivialupperbound} we also have for $\bb \in \ZZ^r$ and $\emptyset \neq Q \subseteq [m]$ that 
\begin{align} \label{eq:nrofsolutions_upperbound}
	\EE{\big| \, \mS_0(\Sub{A}{Q},\bb) \cap [n]_p^{|Q|} \big|} & \leq n^{|Q| - r_Q} \, p^{|Q|} \quad \text{for} \quad \emptyset \neq Q \subseteq [m].
\end{align}

\medskip

Following the alteration method as used for example by Schacht~\cite{Sch12}, we make three case distinctions. For this, we define the \emph{maximum density} of $A$ to be
\begin{equation} \label{eq:maxdensity}
	m(A) = \max_{\emptyset \neq Q \subseteq [m]} \frac{|Q|}{|Q|-r_Q}.
\end{equation}
Note the difference to the previously defined maximum $1$-density. The constant $c = c(A,\epsilon)$ will be stated later in context in Equation~\eqref{eq:smallc}. Note that we need to cover the whole range of $0 \leq p(n) \leq c \, n^{-1/m_1(A)}$ since we are not dealing with a monotone property.

\subsection*{Case 1.} Assume that $p \ll n^{-1/m(A)}$. Let $\emptyset \neq Q_1 \subseteq [m]$ be a set of column indices such that $|Q_1|/(|Q_1|-r_{Q_1}) = m(A)$. By Equation~\eqref{eq:nrofsolutions_upperbound} we now have 
\begin{equation*}
	\lim_{n \to \infty} \EE{| \, \mS_0(\Sub{A}{Q_1}) \cap [n]_p^m|} \leq \lim_{n \to \infty} n^{|Q_1|-r_{Q_1}} \, p^{|Q_1|} = 0.
\end{equation*}
Markov's Inequality and Lemma~\ref{lemma:nosolutions} therefore give us $\lim_{n \to \infty} \PP{| \, \mS_0(A) \cap [n]_p^m| \neq 0} = 0$, see also Rué et al.~\cite{RSZ15}. It clearly follows that we also have $\lim_{n \to \infty}  \PP{[n]_p \to_{\epsilon} A} = 0$ for any $\epsilon > 0$ if $p = p(n) \ll n^{-1/m(A)}$.

\subsection*{Case 2.} Assume that $n^{-1} \ll p \ll n^{-1/m_1(A)}$. Let $Q_2 \subseteq [m]$ be a set of column indices satisfying $|Q_2| \geq 2$ such that $(|Q_2|-1)/(|Q_2|-r_{Q_2}-1) = m_1(A)$ and $|Q_2|$ is as small as possible. Since $np \to \infty$ we have $\lim_{n \to \infty} \PP{|[n]_p| \geq np/2} = 1$ due to Chernoff's bound. The expected number of solutions in $[n]_p$ now is asymptotically smaller than the number of elements since by Equation~\eqref{eq:nrofsolutions_upperbound} we have 
\begin{equation*}
	\EE{| \, \mS_0(\Sub{A}{Q_2}) \cap [n]_p^m|} \leq m^m \, n^{|Q_2|-r_{Q_2}} \, p^{|Q_2|} = m^m \, np \: \big( n^{1/m_1(A)} \, p \big)^{|Q_2|-1} = o(np/2).
\end{equation*}
It follows by Markov's Inequality that for any subset of $[n]_p$ of positive density $\epsilon > 0$ we can remove one element per solution contained in this subset so that the resulting set is free of solutions while asymptotically almost surely still having positive density $\epsilon$ in $[n]_p$. Lemma~\ref{lemma:nosolutions} therefore gives us that we have $\lim_{n \to \infty} \PP{[n]_p \to_{\epsilon}^{\star} A} = 0$ for any $\epsilon > 0$ if $n^{-1} \ll p = p(n) \ll n^{-1/m_1(A)}$.

\subsection*{Case 3.} Lastly, assume that $n^{-1/m(A)} \ll p \leq c n^{-1/m_1(A)}$, where $c = c(A,\epsilon)$ will be given in Equation~\eqref{eq:smallc}. Due to Chernoff we again have $|[n]_p| \geq np/2$ asymptotically almost surely. We now observe that due to Lemma~\ref{lemma:nosolutions} we can replace $A$ with $\Sub{A}{Q_2}$ if necessary in order to guarantee that $(m-1)/(m-\rank{A}-1) = m_1(A)$ as well as $(|Q|-1)/(|Q|-r_Q-1) < m_1(A)$ for any $Q \subsetneq [m]$. We have previously observed that $\Sub{A}{Q_2}$ is again irredundant, positive and abundant. Let $X = (X_n)_{n \in \NN}$ denote the sequence of random variables counting the number of proper solutions in $[n]_p$, that is $X_n = | \, \mS_0(A) \cap [n]_p^m|$ for $n \in \NN$. For
\begin{equation} \label{eq:smallc}
	c = c(A,\epsilon) = \left( \frac{1-\epsilon}{4} \right)^{1/(m-1)}
\end{equation}
it follows by Equation~\eqref{eq:nrofsolutions_upperbound} that 
\begin{equation*}
	\EE{X_n} \leq n^{m-\rank{A}} \, p^m \leq np \, \big( n^{1/m_1(A)} \, p \big)^{m-1} \leq (1-\epsilon) \, np/4.
\end{equation*}
For a given vector $\bx = (x_1, \dots, x_m)$ we let $s(\bx) = \{ x_1, \dots, x_m \}$ denote the set of its entries. Using this we can now estimate the variance of $X_n$ by 
\begin{align*}
	\Var{X_n} -\EE{X_n} & \leq \sum_{\substack{\bx,\by \in \mS_0(A) \\ s(\bx) \cap \, s(\by) \neq \emptyset}} \!\!\!\!\!\!\!  p^{|s(\bx)| + |s(\by)| - |s(\bx) \cap \, s(\by)|} \\
	& = \sum_{\bx \in \mS_0(A)} p^{|s(\bx)|} \, \Bigg( \; \sum_{\emptyset \neq Q \subsetneq [m]} \Bigg[ \sum_{ \substack{\by \in \mS_0(A) \\ s(\bx) \cap s(\by) = s(\by^{\widebar{Q}})} } p^{|s(\by)| - |s(\bx) \cap \, s(\by)|} \; \Bigg] \quad + \!\!\! \sum_{\substack{\by \in \mS_0(A) \\ s(\bx) \subseteq s(\by)}} \!\!\!\! 1 \quad \Bigg) \\
	& \leq \sum_{\bx \in \mS_0(A)} p^{m} \Bigg( \sum_{\emptyset \neq Q \subsetneq [m]} m^{|\widebar{Q}|} \, \max_{\bb \in \ZZ^r} \Bigg( \, \sum_{\by' \in \mS_0(\Sub{A}{Q},\bb)} p^{|Q|} \Bigg) \: + \: m^m \Bigg) \\
	& \leq m^m \sum_{\bx \in \mS_0(A)} p^{m} \left( \sum_{\emptyset \neq Q \subsetneq [m]} \, n^{|Q|-r_Q} p^{|Q|} \: + \: 1 \right) \\
	& = O \Big( n^{m-\rank{A}} p^m \; \max_{\emptyset \neq Q \subsetneq [m]} \big( n^{|Q|-r_Q} p^{|Q|} \big) \Big).
\end{align*}
We observe that due to Equation~\eqref{eq:nrofsolutions_lowerbound} and the assumption on $m_1(A)$ we now have 
\begin{align*}
	\Var{X_n} & = o(\EE{X}^2).
\end{align*}
Chebyshev's inequality therefore gives us $\PP{|X - \EE{X}| \geq \EE{X}} = o(1)$ so that
\begin{equation*}
	\big| \, \mS_0(A) \cap [n]_p^m \big| \leq 2 \, \EE{X} = (1-\epsilon) \, np/2
\end{equation*}
asymptotically almost surely. It follows that, given a set of density $\epsilon$, we can remove one element from $[n]_p$ for each solution in $\mS_0(A) \cap [n]_p^m$ and asymptotically almost surely still be left with a set of density $\epsilon$, so that $\lim_{n \to \infty}  \PP{[n]_p \to_{\epsilon} A} = 0$ for any $\epsilon > 0$ if $n^{-1/m(A)} \ll p = p(n) \leq c \, n^{-1/m_1(A)}$ where $c = c(A,\epsilon)$ as given in Equation~\eqref{eq:smallc}.
\end{proof}

%%% BIBLIOGRAPHY
\bibliography{bib}

\begin{thebibliography}{10}

\bibitem{vdW27}
B.~L. Van~der Waerden, ``{Beweis einer Baudetschen Vermutung},'' {\em Nieuw
  Arch. Wisk}, vol.~15, no.~2, pp.~212--216, 1927.

\bibitem{Sz75}
E.~Szemer{\'e}di, ``On sets of integers containing no {$k$} elements in
  arithmetic progression,'' {\em Acta Arithmetica}, vol.~27, pp.~199--245,
  1975.
\newblock Collection of articles in memory of Juri{\u\i} Vladimirovi{\v{c}}
  Linnik.

\bibitem{Ra33}
R.~Rado, ``{Studien zur Kombinatorik},'' {\em Mathematische Zeitschrift},
  vol.~36, no.~1, pp.~424--470, 1933.

\bibitem{FGR88}
P.~Frankl, R.~L. Graham, and V.~R{\"o}dl, ``Quantitative theorems for regular
  systems of equations,'' {\em Journal of Combinatorial Theory, Series A},
  vol.~47, no.~2, pp.~246--261, 1988.

\bibitem{Sch12}
M.~Schacht, ``Extremal results for random discrete structures,'' {\em Annals of
  Mathematics}, no.~184(2), pp.~333--365, 2016.

\bibitem{RR97}
V.~R{{\"o}}dl and A.~Ruci{{\'n}}ski, ``Rado partition theorem for random
  subsets of integers,'' {\em Proceedings of the London Mathematical Society.
  Third Series}, vol.~74, no.~3, pp.~481--502, 1997.

\bibitem{FRS10}
E.~Friedgut, V.~R{\"o}dl, and M.~Schacht, ``Ramsey properties of random
  discrete structures,'' {\em Random Structures \& Algorithms}, vol.~37, no.~4,
  pp.~407--436, 2010.

\bibitem{NS14}
R.~Nenadov and A.~Steger, ``A short proof of the random ramsey theorem,'' {\em
  Combinatorics, Probability and Computing}, vol.~25, no.~01, pp.~130--144,
  2016.

\bibitem{BMS14}
J.~Balogh, R.~Morris, and W.~Samotij, ``Independent sets in hypergraphs,'' {\em
  Journal of the American Mathematical Society}, vol.~28, pp.~669--709, 2015.

\bibitem{ST15}
D.~Saxton and A.~Thomason, ``Hypergraph containers,'' {\em Inventiones
  mathematicae}, vol.~201, no.~3, pp.~925--992, 2015.

\bibitem{CG10}
D.~Conlon and W.~T. Gowers, ``Combinatorial theorems in sparse random sets,''
  {\em Annals of Mathematics}, no.~184(2), pp.~367--454, 2016.

\bibitem{KSV12}
D.~Kr{\'a}l', O.~Serra, and L.~Vena, ``A removal lemma for systems of linear
  equations over finite fields,'' {\em Israel Journal of Mathematics},
  vol.~187, no.~1, pp.~193--207, 2012.

\bibitem{Rz93}
I.~Z. Ruzsa, ``Solving a linear equation in a set of integers {I},'' {\em Acta
  Arithmetica}, vol.~65, no.~3, pp.~259--282, 1993.

\bibitem{RSZ15}
J.~Ru{\'e}, C.~Spiegel, and A.~Zumalac{\'a}rregui, ``Threshold functions and
  poisson convergence for systems of equations on random sets.'' Available
  online at {\tt arXiv:1212.5496v3}, 2015.

\bibitem{ET41}
P.~Erd{\"o}s and P.~Tur{\'a}n, ``On a problem of sidon in additive number
  theory, and on some related problems,'' {\em Journal of the London
  Mathematical Society}, vol.~1, no.~4, pp.~212--215, 1941.

\bibitem{HT16}
R.~Hancock and A.~Treglown, ``On solution-free sets of integers.'' Available
  online at {\tt arXiv:1607.08399}, 2016.

\bibitem{ES46}
P.~Erd{\"o}s, A.~H. Stone, {\em et~al.}, ``On the structure of linear graphs,''
  {\em Bull. Amer. Math. Soc}, vol.~52, no.~1087-1091, p.~4, 1946.

\bibitem{ES65}
P.~Erd{\H{o}}s and M.~Simonovits, ``A limit theorem in graph theory,'' in {\em
  Studia Sci. Math. Hung}, Citeseer, 1965.

\bibitem{Ro53}
K.~F. Roth, ``On certain sets of integers,'' {\em Journal of the London
  Mathematical Society}, vol.~1, no.~1, pp.~104--109, 1953.

\bibitem{JR11}
S.~Janson and A.~Ruci{\'n}ski, ``Upper tails for counting objects in randomly
  induced subhypergraphs and rooted random graphs,'' {\em Arkiv f{\"u}r
  Matematik}, vol.~49, no.~1, pp.~79--96, 2011.

\bibitem{KRSS16}
C.~Kusch, J.~Rué, C.~Spiegel, and T.~Szabó, ``Random strategies are nearly
  optimal for generalized van der {W}aerden games.'' In preparation.

\bibitem{Va59}
P.~Varnavides, ``On certain sets of positive density,'' {\em Journal of the
  London Mathematical Society. Second Series}, vol.~1, no.~3, pp.~358--360,
  1959.

\bibitem{Gr05}
B.~Green, ``A {S}zemer{\'e}di-type regularity lemma in abelian groups, with
  applications,'' {\em Geometric \& Functional Analysis GAFA}, vol.~15, no.~2,
  pp.~340--376, 2005.

\bibitem{Sh10}
A.~Shapira, ``A proof of {G}reen's conjecture regarding the removal properties
  of sets of linear equations,'' {\em Journal of the London Mathematical
  Society}, p.~jdp076, 2010.

\end{thebibliography}
\bibliographystyle{ieeetr}

\end{document}